\documentclass[a4paper, 11pt]{article}

\usepackage{amsmath,amsthm,amssymb,bbm,enumitem}
\usepackage[margin=2cm]{geometry}
\usepackage{tikz, tkz-graph, mathrsfs}

\theoremstyle{plain}
\newtheorem{theorem}{Theorem}[section]

\newtheorem{conjecture}[theorem]{Conjecture}
\newtheorem{corollary}[theorem]{Corollary}

\newtheorem{lemma}[theorem]{Lemma}

\newtheorem{proposition}[theorem]{Proposition}

\theoremstyle{definition}
\newtheorem{claim}{Claim}
\newtheorem*{claim*}{Claim}

\usetikzlibrary{arrows}

\newcommand{\edge}[2]{\draw (#1) -- (#2);}
\newcommand{\arc}[2]{{\draw[-latex] (#1) edge (#2);}}

\newcommand{\Proba}{\mathbb{P}}

\newcommand{\neighbourhood}{N}
\newcommand{\degree}{d}

\newcommand{\dMax}{\degree^{\max}}

\newcommand{\loops}[1]{#1^\circ}

\newcommand{\feedback}{\tau}

\newcommand{\Fix}{\mathrm{Fix}}

\newcommand{\functions}{\mathrm{F}}

\newcommand{\stability}{\mathrm{s}}

\newcommand{\instability}{\mathrm{i}}

\newcommand{\dH}{d_\mathrm{H}}

\newcommand{\one}[1]{{\mathbbm 1}\left\{ #1 \right\}}

\begin{document}

\title{On the stability and instability of finite dynamical systems with prescribed interaction graphs}
\author{Maximilien Gadouleau\footnote{Department of Computer Science, Durham University, South Road, Durham DH1 3LE, UK. Email: m.r.gadouleau@durham.ac.uk}}
\maketitle

\begin{abstract}
The dynamical properties of finite dynamical systems (FDSs) have been investigated in the context of coding theoretic problems, such as network coding, and in the context of hat games, such as the guessing game and Winkler's hat game. The instability of an FDS is the minimum Hamming distance between a state and its image under the FDS, while the stability is the minimum of the reciprocal of the Hamming distance; they are both directly related to Winkler's hat game. In this paper, we study the value of the (in)stability of FDSs with prescribed interaction graphs. The first main contribution of this paper is the study of the maximum stability for interaction graphs with a loop on each vertex. We determine the maximum (in)stability for large enough alphabets and also prove some lower bounds for the Boolean alphabet. We also compare the maximum stability for arbitrary functions compared to monotone functions only. The second main contribution of the paper is the study of the average (in)stability of FDSs with a given interaction graph. We show that the average stability tends to zero with high alphabets, and we then investigate the average instability. In that study, we give bounds on the number of FDSs with positive instability (i.e fixed point free functions). We then conjecture that all non-acyclic graphs will have an average instability which does not tend to zero when the alphabet is large. We prove this conjecture for some classes of graphs, including cycles.
\end{abstract}

\section{Introduction} \label{sec:intro}

Many entities (such as genes, neurons, persons, computers, etc.) organise themselves as complex networks, where each entity has a finitely valued state and a function which updates the value of the state. Since entities influence each other, this local update function depends on the states of some of the entities. Such a network is called a {\bf Finite Dynamical System} (FDS), with special cases or variants appearing under different names, such as Boolean Networks \cite{Kau69, Tho73}, Boolean Automata Networks \cite{NRS13}, Multi-Valued Networks \cite{BS12}, etc. The main problem when studying an FDS is to determine its dynamics, such as the number of its fixed points, or how the trajectory of a state depends on the initial state.  

FDSs have been used to represent a network of interacting entities as follows. A network of $n$ entities has a state $x= (x_1,\dots, x_n) \in [q]^n$, represented by a variable $x_v$ taking its value in a finite alphabet $[q] = \{0, 1, \dots, q-1\}$ on each entity $v$. The state then evolves according to a deterministic function $f = (f_1,\dots,f_n) : [q]^n \to [q]^n$, where $f_v : [q]^n \to [q]$ represents the update of the local state $x_v$. Although different update schedules have been studied, we are focusing on the parallel update schedule, where all entities update their state at the same time, and $x$ becomes $f(x)$. FDSs have been used to model different real-life networks of entities, such as gene regulatory networks, neural networks, social interactions, etc. (see \cite{Gad} and references therein).

The structure of an FDS $f: A^n \to A^n$ can be represented via its \textbf{interaction graph} $G(f)$, which indicates which update functions depend on which variables. More formally, $G(f)$ has $\{1,\dots,n\}$ as vertex set and there is an arc from $u$ to $v$ if $f_v(x)$ depends on $x_u$. In different contexts, the interaction graph is known (or at least well approximated), while the actual update functions are not. One main problem of research on FDSs is then to predict their dynamics according to their interaction graphs.

Hat games are an increasingly popular topic in combinatorics. Typically, a hat game involves $n$ players, each wearing a hat that can take a colour from a given set of $q$ colours. No player can see their own hat, but each player can see some subset of the other hats. All players are asked to guess the colour of their own hat at the same time. For an extensive review of different hat games, see \cite{Krz12}.

In \textbf{Winkler's hat game}, the team scores as many points as players guessing correctly. The aim is then to construct a guessing function $f$ which guarantees a score for any possible configuration of hats \cite{Win01}. Winkler's hat game was studied in \cite{BHKL08, GG15, Gad, Szc17, Far}. This hat game, and a dual version where the players aim to guess the colour incorrectly, can be formalised in terms of the stability and the instability of FDSs with prescribed interaction graphs \cite{Gad}.

In the variation called the ``\textbf{guessing game}'' \cite{Rii07}, the team wins if everyone has guessed their colour correctly; the aim is to maximise the number of hat assignments which are correctly guessed by all players. This version of the hat game then aims to determine the so-called guessing number of a digraph. Both hat games are linked in \cite{Gad}. The guessing number is itself related to the network coding solvability problem \cite{Rii06, Rii07, Rii07a}.

Previous work on the (in)stability of FDSs were actually concerned with functions whose interaction graphs are contained in a given digraph. In this paper, we focus on ``strict'' (in)stability, where the interaction graph of the FDS is exactly a given digraph $D$. This paper has two main contributions.
\begin{itemize}
\item Firstly, we study the maximum (in)stability of an FDS with a prescribed interaction graph. We focus on graphs with a loop on each vertex, as this is where the non-strict case is trivial. The stability is more interesting; we give its limit for large alphabets and we give a tight lower bound for the Boolean case. We finally focus on the maximum stability of monotone Boolean FDSs.

\item Secondly, we study the average (in)stability  of an FDS with a prescribed interaction graph. This time, the average stability is relatively easy to handle, while the instability is more interesting. We relate the latter problem to the number of fixed-point free functions with a prescribed interaction graph. We conjecture that the number of fixed-point free elements does not tend to zero for large alphabets, unless the interaction graph is acyclic. We finally prove a stronger version of that conjecture for cycles. 
\end{itemize}

The rest of the paper is as follows. Section \ref{sec:stability_definition} review some background on graphs and FDSs and introduces the strict (in)stability of digraphs. Section \ref{sec:max_s} gives our results on the maximum (in)stability of FDSs with a given loop-full interaction graph. Section \ref{sec:average_i} then investigates the average (in)stability of FDSs with a given loop-full interaction graph.

\section{Stability and instability: definitions and basic properties} \label{sec:stability_definition}

\subsection{Finite dynamical systems}

Let $n$ be a positive integer and $V = \{1, \dots, n\}$. A graph $D$ on $V$ is a pair $D = (V,E)$ with $E \subseteq V^2$ (in other words, all our graphs are directed). Paths and cycles are always directed. The girth of $D$ is the minimum length of a cycle in $D$. A feedback vertex set is a set of vertices $I$ such that $D-I$ has no cycles. The minimum size of a feedback vertex set is denoted $\feedback(D)$. If $J \subseteq V$, then $D[J]$ is the subgraph of $D$ induced by $J$. If this graph is acyclic, then the vertices of $J$ can be sorted in acyclic ordering: $J = \{j_1, \dots, j_k\}$ where $(j_a, j_b) \in E$ only if $a < b$. The in-neighbourhood of a vertex $v$ in $D$ is denoted by $\neighbourhood(v;D)$ and its in-degree is denoted by $\degree(v;D)$; we may simply write $\neighbourhood(v)$ and $\degree(v)$ if the graph is clear from the context. A vertex with empty in-neighbourhood is a source of $D$. We denote the maximum in-degree of $D$ as $\dMax(D)$.

A loop is an arc of the form $(v,v)$. We view a loop as a cycle of length one. If $D$ has no loops, we say it is loopless; if $D$ has a loop on every vertex, we say it is loop-full.  If $D$ is a loopless graph, then $\loops{D}$ is the loop-full graph obtained by taking $D$ and adding a loop to very vertex. An edge $uv$ is a pair of arcs $\{(u,v), (v,u)\}$ for $u \ne v$; a graph is simple if its arc set can be partitioned into edges (i.e. the graph is loopless and symmetric).

Let $q \ge 2$, we denote $[q]=\{0,1,\dots,q-1\}$. For all $x = (x_1, \dots, x_n) \in [q]^n$, we use the following shorthand notation for all $J = \{j_1, \dots, j_k\} \subseteq V$: $x_J = (x_{j_1}, \dots, x_{j_k})$. For all $x,y\in[q]^n$ we set $\Delta(x,y):=\{i\in [n]\,:\,x_i\ne y_i\}$. The Hamming distance between $x$ and $y$ is $\dH(x,y)=|\Delta(x,y)|$. Finally, for any property $\mathcal{P}$, we denote the function which returns $1$ if $\mathcal{P}$ is satisfied and $0$ otherwise by $\one{\mathcal{P}}$. For instance, $\dH(x, y) = \sum_{v=1}^n \one{x_v \ne y_v}$.

Let $f : [q]^n \to [q]^n$ be a \textbf{Finite Dynamical System} (FDS). We write the FDS as $f = (f_1, \dots, f_n)$ where $f_v : [q]^n \to [q]$ is a local function of $f$.  We also use the shorthand notation $f_J : [q]^n \to [q]^{|J|}$, $f_J = (f_{j_1}, \dots, f_{j_k})$. We associate with $f$ the graph $G(f)$, referred to as the \textbf{interaction graph} of $f$, defined by: the vertex set is $V$; and for all $u,v\in V$, there exists an arc $(u,v)$ if and only if $f_v$ depends essentially on $x_u$, i.e. there exist $x,y\in [q]^n$ that only differ by $x_u\neq y_u$ such that $f_v(x)\neq f_v(y)$. For a graph $D$, we denote by $\functions(D,q)$ the set of FDSs $f : [q]^n \to [q]^n$ with $G(f) \subseteq D$ and by $\functions[D,q]$ the set of FDSs $f : [q]^n \to [q]^n$ with $G(f) = D$. We also denote the set of all FDSs $f : [q]^n \to [q]^n$ as $\functions(n,q)$. When dealing with $f \in \functions(n,q)$ we shall implicitly do all operations modulo $q$.

The \textbf{stability} and \textbf{instability} of an FDS $f$ are respectively given by \cite{Gad}
\begin{align*}
	\stability(f) &:= \min_{x\in[q]^n} \left\{ n - \dH(x,f(x)) \right\},\\
	\instability(f) &:=\min_{x\in[q]^n} \dH(x,f(x)).
\end{align*}
Then $\stability(f) = n$ if and only if $f$ is the identity, while $\instability(f) = 0$ if and only if $f$ has a fixed point. By definition, we have $\stability(f) + \instability(f) \le n$. It is easy to prove that the minimum (in)stability in $\functions[D,q]$ is equal to zero for all $D$ and all $q$.

\subsection{Stability and instability: general properties} \label{sec:instability_general}

The \textbf{$q$-stability} and \textbf{$q$-instability} of a graph $D$ are \cite{Gad}
\begin{align*}
	\stability(D,q) &:= \max_{f \in \functions(D,q)} \stability(f),\\
	\instability(D,q) &:= \max_{f \in \functions(D,q)} \instability(f).
\end{align*}

We give a hat game intuition for these quantities. Consider a game played by a team of $n$ players, each having a hat that can take $q$ possible colours. Each player can only see a subset of the hats; let $D$ be the graph on $\{1, \dots, n\}$ such that $(u,v)$ is an arc if and only if $v$ can see $u$'s hat. Before the game starts, the players come up with a common strategy for guessing their own hat's colour; no communication is allowed once the colours have been chosen. Let $x \in [q]^n$ represent a configuration of hats. The team's guess is then denoted $f(x)$, where $f_v$ depends on $x_u$ only if $(u,v) \in D$, or in other words $G(f) \subseteq D$. Therefore, exactly $n - \dH(x, f(x))$ players will guess correctly if the hat configuration is $x$. If the aim of the game is to find a guessing strategy that maximises the number of correct guesses for any possible configuration of hats, then the maximum is given by $\stability(D,q)$. If on the other hand, the aim is to maximise the number of incorrect guesses, then the solution is $\instability(D,q)$.

Similarly, we denote the strict $q$-stability and the strict $q$-instability of $D$ as
\begin{align*}
	\stability[D,q] &:= \max_{f \in \functions[D,q]} \stability(f),\\
	\instability[D,q] &:= \max_{f \in \functions[D,q]} \instability(f).
\end{align*}
The hat game intuition here is not so straightforward, as we force each player to be influenced by every hat that they see at least once. For the stability, this makes a big difference when there is a loop on a vertex $v$. This case corresponds to player $v$ being able to see their own hat, hence obviously the optimal strategy is $f_v(x) = x_v$. However, in the strict case, if $(u,v) \in D$, then at some point $x_u$ has to influence $f_v$, which forces $v$ to guess incorrectly. The game is now about minimising the negative influence of other players. We shall study this more in detail in the next section.

Some basic properties of the (in)stability are given below.

\begin{lemma} \label{lem:i_increasing}
For every loopless graph $D$ and any $q \ge 2$, the following hold.
$$
	\stability[D,2] = \instability[D,2] \le \instability[D,q] \le \instability(D,q) \le \instability[D,2q].
$$
\end{lemma}

\begin{proof}
The relations $\stability[D,2] = \instability[D,2] \le \instability[D,q]$ were implicitly proved in \cite{Gad}. Moreover, $\instability[D,q] \le \instability(D,q)$ by definition. We now prove the last inequality. We view an element $y$ of $[2q]$ as $(y^1, y^2) \in [q] \times [2]$ and we let $f \in \functions(D,q)$ with maximum instability. We construct $g \in \functions[D,2q]$ as follows:
$$
	g_v \left( \left( x^1, x^2 \right) \right) = \left( f_v(x^1) , \bigwedge_{u \in \neighbourhood(v)} x_u^2 \right).
$$
Then it is clear that $\instability(g) \ge \instability(f)$.
\end{proof}

In \cite{Gad}, it is proved that for any loopless $D$ and any $q \ge 2$, the instability $\instability(D,q)$ is a non-decreasing function of $q$; the proof actually shows that the instability $\instability[D,q]$ is non-decreasing. It is also proved in \cite{Gad} that the stability $\stability(D, q)$ is a non-increasing function of $q$; however, it is unknown whether the strict stability is also non-decreasing.

\begin{corollary}
For all $q$ sufficiently large, $\instability[D,q] = \feedback(D)$.
\end{corollary}

\begin{proof}
According to \cite{Gad}, $\instability(D,q) = \feedback(D)$ for all $q \ge 2^n$. Therefore, for all $q \ge 2^{n+1}$, Lemma \ref{lem:i_increasing} and the fact that the $q$-strict instability is non decreasing yield 
$$
	\feedback(D) = \instability(D, \lfloor q/2 \rfloor) \le \instability[D, 2\lfloor q/2 \rfloor] \le \instability[D,q] \le \feedback(D).
$$
\end{proof}

Moreover, considering the function $f \in \functions[K_n, q]$ defined as
$f_v(x) = v - \sum_{u \ne v} x_u $
and combining with \cite[Proposition 3]{Gad}, we easily obtain that for every $n$ and $q\ge 2$, 
\begin{align*}
	\stability[K_n, q]  = \stability(K_n, q) &= \left\lfloor \frac{n}{q} \right\rfloor,\\
	\instability[K_n,q] = \instability(K_n, q)  &= n - \left\lceil \frac{n}{q} \right\rceil.
\end{align*}
These are the highest values of the (in)stability for all loopless graphs.

\section{Strict (in)stability for loop-full graphs} \label{sec:max_s}

\subsection{Exact results}

We now consider the strict (in)stability of loop-full graphs. Let $D$ be a loopless graph and $\loops{D}$ be obtained by adding a loop on every vertex. It is easily seen that the strict $q$-(in)stability of $\loops{D}$ is a non-decreasing function of $q$. Then let $S$ be the set of sources of $D$ and define
$$
	\sigma(D) := \max \left\{ |U| : U \subseteq V \setminus S, \neighbourhood(u; \loops{D}) \cap \neighbourhood(v; \loops{D}) = \emptyset \,\, \forall \, u, v \in U, u \ne v \right\}.
$$
This can be viewed as the independence number of a related simple graph $G = (V, E')$, where $u$ and $v$ are adjacent if and only if either $(u,v) \in E$ or $(v,u) \in E$ or there exists $w \in \neighbourhood(u; D) \cap \neighbourhood(v; D) $.

We first determine the values of the strict (in)stability for large enough alphabets.

\begin{proposition} \label{prop:s_loops}
For any loopless graph $D$, 
\begin{itemize}	
	\item $\stability[\loops{D}, q] = n - \sigma(D)$ for all $q \ge n - |S|$;
	
	\item $\instability[\loops{D},q] = n$ for all $q \ge 3$;
\end{itemize}
\end{proposition}

\begin{proof}
All the neighbourhoods are with respect to $\loops{D}$. Let $\sigma = \sigma(D)$ and $U = \{u_1, \dots, u_\sigma\} \subseteq V \setminus S$, such that $\neighbourhood(u_i) \cap \neighbourhood(u_j) = \emptyset$ for all $1 \le i, j \le \sigma$. For any $1 \le i \le \sigma$, there exists $y_{\neighbourhood(u_i)}$ such that 
$$
	f_{u_i}(y_{\neighbourhood(u_i)}) \ne y_{u_i}.
$$
Therefore, the state $y = (y_{\neighbourhood(u_1)}, \dots, y_{\neighbourhood(u_\sigma)}, y_T)$, where $T$ is the rest of the vertices, satisfies $\dH(y, f(y)) \ge \sigma$. 

Conversely, without loss suppose the set of loops in $D$ is $S = \{n - |S| + 1, \dots, n\}$. Then let $q \ge n - |S|$ and consider the following function $f \in \functions[\loops{D}, q]$:
\begin{align*}
	f_s(x) &= x_s  	&& \forall s \in S,\\
	f_v(x) &= x_v + \one{ x_{\neighbourhood(v)} = (v - 1, \dots, v - 1) } && \forall v \notin S.
\end{align*}
Fix $x \in [q]^n$ and let $J = \Delta(x, f(x))$. Firstly, $J \subseteq V \setminus S$. Now, let $u,v \in J$ such that $\neighbourhood(u) \cap \neighbourhood(v) \ne \emptyset$, say $w$ is in the intersection. Then $x_w = u - 1 = v - 1$, which is impossible. Therefore, for all $u,v \in J$, $\neighbourhood(u) \cap \neighbourhood(v) = \emptyset$ and hence $\dH(x, f(x)) = |J| \le \sigma$.

For the instability, let $q \ge 3$ and $f \in \functions[\loops{D}, q]$ such that
\begin{align*}
	f_s(x) &= x_s + 1 && \forall s \in S,\\
	f_v(x) &= x_v + 1 + \one{ x_{\neighbourhood(v)} = (0, \dots, 0) } && \forall v \notin S.
\end{align*}
Then clearly $\instability(f) = n$.
\end{proof}

We remark that the condition $q \ge n - |S|$ in Proposition \ref{prop:s_loops} is tight for some graphs. Indeed, let $D$ be an out-star, i.e. $D = (V, E)$ with $E = \{ (1, v) : 2 \le v \le n \}$, then $|S| = 1$, $\sigma(D) = 1$ and $\stability[D,q] \le n - 2$ for all $q \le n-2$.

\begin{proposition}
For any graph $D$ without sources, the strict stability of $\loops{D}$ and its maximum in-degree are related by
\[
	q^{\dMax(\loops{D})} \ge \frac{n}{n-\stability[\loops{D},q]}.
\]
Conversely, for any $t \ge 1$, $\Delta \ge 1$ and $q$, there exists an infinite family of graphs with $\stability[\loops{D},q] = n - t$, $n = t q^\Delta$ and $\dMax(D) = \Delta + 1$.
\end{proposition}

\begin{proof}
For any $f \in \functions[\loops{D},q]$ and any $v$, there is at least one value of $x_{\neighbourhood(v)}$ such that $f_v(x_{\neighbourhood(v)}) \ne x_v$. Therefore, there are at least $q^{n - d(v)}$ states $x$ such that $f_v(x) \ne x_v$. This yields
\begin{align*}
	T &:= \sum_{x \in [q]^n} \dH(x, f(x))\\
	&= \sum_{v \in V} \sum_{x \in [q]^n} \one { f_v(x) \ne x_v }\\
	&\ge \sum_{v \in V} q^{n - d(v)}\\
	&\ge n q^{n - \dMax(\loops{D})}.
\end{align*}
On the other hand, $T \le q^n (n - \stability[\loops{D}, q])$. Combining, we obtain
\[
	q^{\dMax(\loops{D})} \ge \frac{n}{n-\stability[\loops{D},q]}.
\]

Now let $D = (V,E)$ on $n = t q ^\Delta$ vertices be as follows. Let $V = C_1 \cup \dots \cup C_t$, where $C_i = \{ c_{i, 1}, \dots, c_{i, q^\Delta} \}$ for all $1 \le i \le t$ and the arc set of $D$ is all possible arcs from $A := \{c_{1,1}, \dots, c_{1,\Delta}\}$: $E = \{ (a, v) : a \in A, v \in V \setminus \{a\} \}$. 

Then let $f \in \functions[\loops{D},q]$ be defined as follows. Denote the set of possible values of $x_A$ as $z_1, \dots, z_{q^\Delta}$, then
$$
	f_{c_{i,j}}(x) = x_{c_{i,j}} + \one { x_A = z_j }.
$$
It is clear that for any $x$ and any $i$, exactly one vertex from $C_i$ guesses wrong, thus $\stability(f) = n-t$.
\end{proof}

\subsection{Lower bound for the Boolean case}

First of all, we make an important remark which characterises the functions with the highest stability in $\functions[\loops{D}, q]$. For any $f_v(x_{\neighbourhood(v)})$, let $\Xi (f_v) := \{ x_{\neighbourhood(v)} : f_v(x_{\neighbourhood(v)}) \ne x_v \}$. Then it is clear that if $\Xi(f_v) \subseteq \Xi (f'_v)$ for all $v$, then $\stability(f) \ge \stability(f')$. In particular, the functions in $\functions[\loops{D}, 2]$ with the highest stability all have $|\Xi(f_v)| = 0$ if $v$ is a source of $D$ (i.e. $f_v(x) = x_v$) and $|\Xi(f_v)| = 1$ otherwise.

\begin{theorem} \label{th:strict_stability_2}
For any loopless graph $D$, $\stability[\loops{D}, 2] \ge n/2$.
\end{theorem}

The proof goes bottom-up, by proving the result for larger and larger classes of graphs. Firstly, 
an \textbf{out-cycle} consists of a cycle (or a single vertex) to which are appended some outgoing arcs, no more than one per vertex. More formally, a graph $D = (V, E)$ is an out-cycle if there exist $k \ge 1$, $0 \le l \le k$, and three sets $A = \{a_1, \dots, a_k\}$, $B = \{a_{b_1}, \dots, a_{b_l}\} \subseteq A$ and $C = \{c_1, \dots, c_l\}$ such that the following hold:
\begin{itemize}
	\item $A \cap C = \emptyset$ and $A \cup C  = V$;
	
	\item if $k \ge 2$, then $D[A]$ is a cycle, say $E(D[A]) = \{(a_i, a_{i+1}) : 1 \le i \le k\}$ with indices computed modulo $k$; if $k=1$, then $D[A]$ is trivial;
	
	\item $E = E(D[A]) \cup \{(a_{b_j}, c_j) : i=1, \dots, l\}$.
\end{itemize}
An example with $k=6$ and $l=3$ is given in Figure \ref{fig:out-cycle}.

\begin{figure}[ht]
\centering
\begin{tikzpicture}[scale = 1.5, vertex/.style={circle, 
  draw, 
  fill=black,
  inner sep=0.08cm}]
  
  \node [vertex] (a1) at (0,0) {};
  \node [vertex] (a2) at (1,0) {};
  \node [vertex] (a3) at (2,0) {};
  \node [vertex] (a4) at (3,0) {};
  \node [vertex] (a5) at (4,0) {};
  \node [vertex] (a6) at (5,0) {};

	\node (aa1) at (-0.6,0.2) {$a_1 = a_{b_1}$};     	
	\node (aa2) at (1,0.2) {$a_2$};
	\node (aa3) at (2,0.2) {$a_3 = a_{b_2}$};
	\node (aa4) at (3,0.2) {$a_4 = a_{b_3}$};
	\node (aa5) at (4,0.2) {$a_5$};
	\node (aa6) at (5.4,0.2) {$a_6$};
	
	\node [vertex] (c1) at (0,-1) {};
	\node [vertex] (c2) at (2,-1) {};
	\node [vertex] (c3) at (3,-1) {};

	\node (cc1) at (0,-1.2) {$c_1$};
	\node (cc2) at (2,-1.2) {$c_2$};
	\node (cc3) at (3,-1.2) {$c_3$};
	
	\arc{a1}{a2}
	\arc{a2}{a3}
	\arc{a3}{a4}
	\arc{a4}{a5}
	\arc{a5}{a6}
	\draw[-latex] (a6) to [bend right] (a1);
	\arc{a1}{c1}
	\arc{a3}{c2}
	\arc{a4}{c3}
\end{tikzpicture}
\caption{An out-cycle with $k=6$ and $l=3$.} \label{fig:out-cycle}
\end{figure}
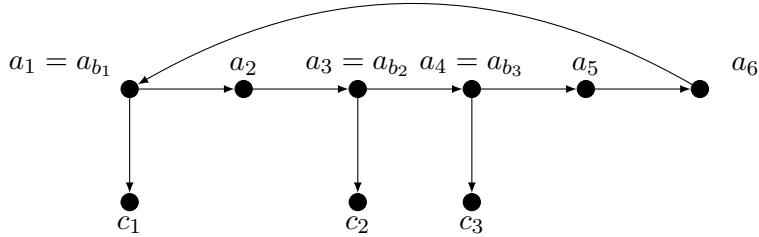

\begin{lemma} \label{lem:out-cycle}
If $D$ is an out-cycle, then $\stability[\loops{D}, 2] \ge n/2$.
\end{lemma}

\begin{proof}
Let $D$ be an out-cycle, with the sets $A$, $B$ and $C$ as above. Then let $f \in \functions[\loops{D}, 2]$ be defined as follows:
\begin{align*}
	f_{a_{b_j+1}}(x) &= x_{a_{b_j}} + (x_{a_{b_j + 1}} + 1) (x_{a_{b_j}} + 1)\\
	&= x_{a_{b_j + 1}} \lor \neg x_{a_{b_j}},  && \forall a_{b_j}  \in B,\\
	f_{a_i}(x) &= x_{a_i} + (x_{a_i} + 1) x_{a_{i-1}}\\
	&= x_{a_i} \lor x_{a_{i-1}},  && \text{if } a_{i-1} \notin B,\\
	f_{c_j}(x) &= x_{c_j} + x_{c_j} x_{a_{b_j}}\\
	&= x_{c_j} \land \neg x_{a_{b_j}} && \forall c_j \in C.
\end{align*}

We now construct the simple graph $D' = (V, E')$ as follows. Its edge set is
$$
	E' = \{ a_{b_j} c_j, c_j a_{b_j + 1} : 1 \le j \le l \} \cup \{ a_i a_{i+1} : a_i \notin B \},
$$
and hence it is a Hamiltonian cycle. An example of $D'$ is given in Figure \ref{fig:D'}. We prove that for any $x \in [2]^n$, $J := \Delta(x, f(x))$ is an independent set of $D'$.
\begin{itemize}
	\item If $a_{b_j} \in J$, then $x_{a_{b_j}} = 0$ hence $c_j \notin J$. Thus there is no edge of the form $a_{b_j} c_j$ in $J$.
	
	\item If $c_j \in J$, then $x_{a_{b_j}} = 1$ hence $a_{b_{j+1}} \notin J$. Thus there is no edge of the form $c_j a_{b_j + 1}$ in $J$.
	
	\item If $a_{i+1} \in J$, where $a_i \notin B$, then $x_{a_i} = 1$ hence $x_{a_i} \notin J$. Thus there is no edge of the form $a_i a_{i+1}$ in $J$.
\end{itemize}
Thus, $\dH(x, f(x)) \le n/2$.
\end{proof}

\begin{figure}[ht]
\centering
\begin{tikzpicture}[scale = 1.5, vertex/.style={circle, 
  draw, 
  fill=black,
  inner sep=0.08cm}]
  
  \node [vertex] (a1) at (0,0) {};
  \node [vertex] (a2) at (1,0) {};
  \node [vertex] (a3) at (2,0) {};
  \node [vertex] (a4) at (3,0) {};
  \node [vertex] (a5) at (4,0) {};
  \node [vertex] (a6) at (5,0) {};

	\node (aa1) at (-0.6,0.2) {$a_1 = a_{b_1}$};     	
	\node (aa2) at (1,0.2) {$a_2$};
	\node (aa3) at (2,0.2) {$a_3 = a_{b_2}$};
	\node (aa4) at (3,0.2) {$a_4 = a_{b_3}$};
	\node (aa5) at (4,0.2) {$a_5$};
	\node (aa6) at (5.4,0.2) {$a_6$};
	
	\node [vertex] (c1) at (0,-1) {};
	\node [vertex] (c2) at (2,-1) {};
	\node [vertex] (c3) at (3,-1) {};

	\node (cc1) at (0,-1.2) {$c_1$};
	\node (cc2) at (2,-1.2) {$c_2$};
	\node (cc3) at (3,-1.2) {$c_3$};
	
	\edge{a1}{c1}
	\edge{c1}{a2}
	\edge{a2}{a3}
	\edge{a3}{c2}
	\edge{c2}{a4}
	\edge{a4}{c3}
	\edge{c3}{a5}
	\edge{a5}{a6}
	\draw (a6) to [bend right] (a1);
\end{tikzpicture}
\caption{The simple graph $D'$ corresponding to $D$ in Figure \ref{fig:out-cycle}} \label{fig:D'}
\end{figure}
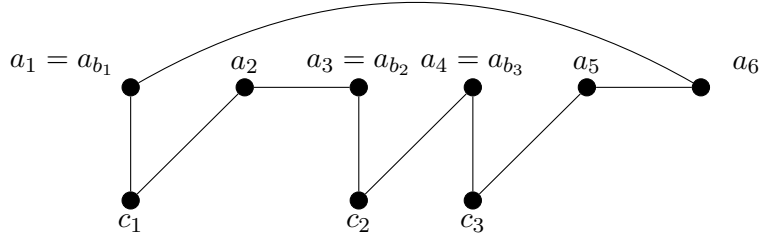

Secondly, let $D = (V, E)$ be a loopless graph and let $a \in V$. We then introduce two ways to append two vertices $u$ and $v$ to $a$. First, \textbf{forking} corresponds to the case where $u$ and $v$ are both connected to $a$: $D_a^F = (V \cup \{u,v\}, E \cup \{ (a, u), (a, v)\})$. Second, \textbf{branching} corresponds to the case where $a, u, v$ form a path: $D_a^B = (V \cup \{u,v\}, E \cup \{ (a, u), (u, v)\})$.

\begin{lemma} \label{lem:A}
If $D$ is obtained from $H$ by branching or forking, then $\stability[\loops{D}, 2] \ge \stability[\loops{H}, 2] + 1$.
\end{lemma}

\begin{proof}
Let $f \in \functions[\loops{H}, 2]$ with $\stability(f) \ge n/2$. Firstly, if $D = H_a^F$, then define $g \in \functions[\loops{D}, 2]$ as
\begin{align*}
	g_i(x) &= f_i(x) && \forall i \in V(H)\\
	g_u(x) &= x_u + (x_a + 1) x_u\\
	&= x_u \land x_a\\
	g_v(x) &= x_v + x_a (x_v + 1)\\
	&= x_v \lor x_a.
\end{align*}
Then, depending on the value of $x_a$, either $g_u(x) = x_u$ or $g_v(x) = x_v$ and we obtain $\stability(g) = \stability(f) + 1$. 

Secondly, if $D = H_a^F$, then define $g \in \functions[\loops{D}, 2]$ as
\begin{align*}
	g_i(x) &= f_i(x) && \forall i \in V(H)\\
	g_u(x) &= x_u + (x_a+1) x_u\\
	&= x_u \land x_a\\
	g_v(x) &= x_v + (x_u + 1) x_v\\
	&= x_v \land x_u.
\end{align*}
Then, depending on the value of $x_u$, either $g_u(x) = x_u$ or $g_v(x) = x_v$ and we obtain $\stability(g) = \stability(f) + 1$.
\end{proof}

A graph $H$ is \textbf{co-functional} if $\dMax(H) \le 1$.

\begin{lemma} \label{lem:B}
Any connected co-functional graph can be obtained from an out-cycle by repeatedly forking and branching.
\end{lemma}

\begin{proof}
We prove it by induction on the number of vertices. It is clearly true for up to three vertices; suppose it is true for $n-2$. For any connected co-functional graph $D$, let $A$ denote the set of vertices in the only cycle of $D$ and let $S$ denote the set of vertices whose out-neighbourhood consists of only leaves. 
\begin{itemize}
	\item Case 1: there exists $s \in S$ with out-degree at least two. Referring to two out-neighbours of $s$ as $u$ and $v$, we see that $D = (D \setminus \{u, v\})_s^F$, and by induction hypothesis, $D$ can be obtained by successive forking and branching.
	
	\item Case 2: all vertices in $S$ have out-degree one and there exists an element $s \in S \setminus A$. Denoting its out-neighbour as $v$ and its in-neighbour as $a$, then $D = (D \setminus \{s, v\})_a^B$, and by induction hypothesis, $D$ can be obtained by successive forking and branching. 
	
	\item Case 3: all vertices in $S$ have out-degree one and $S \subseteq A$. Then $D$ is an out-cycle.
\end{itemize}
\end{proof}

We can now prove the theorem for co-functional graphs.

\begin{lemma} \label{lem:C}
For any co-functional graph $D$, $\stability[\loops{D}, 2] \ge n/2$.
\end{lemma}

\begin{proof}
If $D$ is connected, then by Lemmas \ref{lem:out-cycle}, \ref{lem:A} and \ref{lem:B}, $\stability[\loops{D}, 2] \ge n/2$. Let $D$ be a co-functional graph with connected components $C_1, \dots, C_s$, with $n_1, \dots, n_s$ vertices respectively. Then $C_i$ is co-functional for all $1 \le i \le s$, and hence 
$$
	\stability[\loops{D}, 2] = \sum_{i=1}^s \stability[\loops{C_i}, 2] \ge \sum_{i=1}^s \frac{n_i}{2} \ge \frac{n}{2}.
$$
\end{proof}

For any loopless $H$ and $D$, we write $H \le D$ if $H$ is a spanning subgraph of $D$ and for all $v \in V$, $v$ is a source in $H$ only if $v$ is a source in $D$.

\begin{lemma} \label{lem:D}
If $H \le D$, then $\stability[\loops{H}, 2] \le \stability[\loops{D}, 2]$.
\end{lemma}

\begin{proof}
Let $f \in \functions[\loops{H}, 2]$. Firstly, if $v$ is a source (of $H$ or of $D$, which is equivalent), then $f_v(x) = x_v + \epsilon_v$ for some $\epsilon_v \in [2]$. Secondly, if $u$ is not a source, then $f_u(x) = x_u + g_u(x_{\neighbourhood(u; \loops{H})})$ for some function $g_u$. We then introduce the function $f' \in \functions[\loops{D}, 2]$, defined as
\begin{align*}
	f'_v(x) &= x_v && \text{if } v \text{ is a source},\\
	f'_u(x) &= x_u + g_u(x_{\neighbourhood(u; \loops{H})}) \land \bigwedge_{a \in \neighbourhood(u; \loops{D}) \setminus \neighbourhood(u; \loops{H})} x_a && \text{if } u \text{ is not a source}.
\end{align*}
Therefore, for any $x \in [2]^n$, we have $f'_v(x) = x_v$ for all sources $v$ and $f'_u(x) \ne x_u$ only if $f_u(x) \ne x_u$ otherwise. Thus $\Delta(x, f'(x)) \subseteq \Delta(x, f(x))$ and $\stability(f') \ge \stability(f)$.
\end{proof}

\begin{lemma} \label{lem:E}
For any loopless graph $D$, there exists a co-functional graph $H$ such that $H \le D$.
\end{lemma}

\begin{proof}
Let $C_1, \dots, C_s$ be the initial strong components of $D$, and let $v_1, \dots, v_s$ such that $v_i \in C_i$ for all $i$. We construct $H$ as follows. For $i$ from $1$ to $s$, do the following
\begin{enumerate}
	\item Use depth-first-search from $v_i$ to construct an out-tree $H_i$.
	
	\item If $v_i$ is not a source in $D$, then add the arc $(u_i, v_i)$, where $u_i$ is a randomly chosen vertex of $\neighbourhood(v_i; D)$.
	
	\item Remove the vertex set of $H_i$.
\end{enumerate}
Then $H = H_1 \cup \dots \cup H_s$.
\end{proof}

The theorem then follows from the last three lemmas.

\subsection{Strict monotone stability}

There is a natural partial order on $[2]^n$, whereby $x \le y$ if and only if $x_v \le y_v$ for all $v$. A function $f \in \functions(n,2)$ is monotone if it preserves that order: $x \le y$ implies $f(x) \le f(y)$. Monotone FDSs have very interesting properties, for instance the celebrated Knaster-Tarski theorem states that the set of fixed points of a monotone FDS forms a non-empty lattice. As such, these FDSs have been studied in different works \cite{Gol85, ARS}.

We denote the set of monotone functions in $\functions[\loops{D}, 2]$ as $\functions_+[\loops{D}, 2]$ and we denote 
$$
	\stability_+[\loops{D}, 2] := \max_{f \in \functions_+[\loops{D},2]} \stability(f).
$$

First of all, we make some important remarks about functions in $\functions_+[\loops{D}, 2]$ with highest stability. Again, such a function satisfies $f_v(x) = x_v$ if $v$ is a source of $D$. For a given $v$ and non-empty $\neighbourhood(v)$, the only monotone coordinate functions $\phi$ and $\psi$ which achieve $|\Xi(\phi)| = |\Xi(\psi)| = 1$ are the following:
\begin{align*}
	\phi(x) &= x_v + \one{ \left( x_v = 0 \right) \land \left( x_{\neighbourhood(v)} = (1, \dots, 1) \right) }\\
		&= x_v \lor \bigwedge_{u \in \neighbourhood(v)} x_u,\\
	\psi(x) &= x_v + \one{ \left( x_v = 1 \right) \land \left( x_{\neighbourhood(v)} = 0 \right) }\\
		&= x_v \land \bigvee_{u \in \neighbourhood(v)} x_u.
\end{align*}
(Obviously, the in-neighbourhoods are with respect to $D$.) Hence for any function $f \in \functions_+[\loops{D}, 2]$ with maximum stability, its local functions are either $\phi$ or $\psi$ as above (if $v$ is not a source of $D$). In that case, we shall write $f_v \equiv \phi$ or $f_v \equiv \psi$, respectively. Moreover, let $\tilde{f}$ be the dual of $f$, defined by $\tilde{f}(x) = \neg f(\neg x)$. Then $\dH(x, f(x)) = \dH(\neg x, \tilde{f}(\neg x))$, hence $\stability(\tilde{f}) = \stability(f)$. Thus, we can always assume that, for a given $v$, $f$ with maximum stability has $f_v \equiv \phi$.

\begin{theorem} \label{th:strict_monotone_stability_2}
For any loopless graph $D$, $\stability_+[\loops{D}, 2] \ge \lfloor n/2 \rfloor$.
\end{theorem}

\begin{proof}
The structure of the proof is the same as that of Theorem \ref{th:strict_stability_2}.  We first give the analogue of Lemma \ref{lem:out-cycle}.

\begin{claim} \label{claim:out-cycle_monotone}
If $D$ is an out-cycle, then $\stability_+[\loops{D},2] \ge \lfloor n/2 \rfloor$.
\end{claim}

\begin{proof}
We use the same notation as for Lemma \ref{lem:out-cycle}. First of all, if $D$ is a cycle, then the function where $f_v \equiv \phi$ for all vertices $v$ has stability $\lceil n/2 \rceil$ (it was already used in Lemma \ref{lem:out-cycle}). Henceforth, we assume that $D$ is not a cycle. Without loss, let $b_1 = 1$, and then define $f \in \functions_+[\loops{D}, 2]$ recursively as follows. We denote $u \equiv v$ if $f_u \equiv \phi \equiv f_v$ or $f_u \equiv \psi \equiv f_v$.
\begin{align*}
	f_{a_1} 	&\equiv \phi 		&& (f_{a_1}(x) = x_{a_1} \lor x_{a_k})\\
	c_j 		&\equiv a_{b_j} 	&& 1 \le j \le l\\
	a_{b_j + 1} &\not\equiv c_j 	&& 1 \le j \le l\\
	a_{i+1} 	&\equiv a_i 		&& i \le k - 1, a_i \notin B.
\end{align*}
Define the simple graph $D^* = (V, E^*)$ as follows. Its edge set is
$$
	E^* = \{ a_{b_j} c_j, c_j a_{b_j + 1} : 1 \le j \le l, b_j \le k-1 \} \cup \{ a_i a_{i+1} : i \le k-1, a_i \notin B \},
$$
and hence it is a Hamiltonian path. An example of $D^*$ is given in Figure \ref{fig:D'}. Again, we can prove that for any $x \in [2]^n$, $\Delta(x, f(x))$ is an independent set of $D^*$, thus $\stability(f) \ge \lfloor n/2 \rfloor$.
\end{proof}

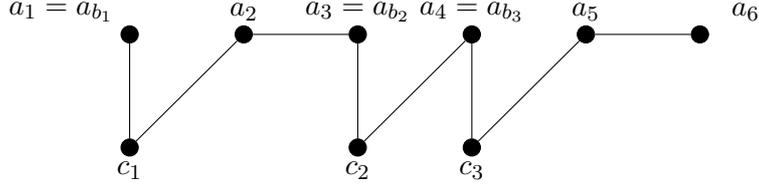
\begin{figure}[ht]
\centering
\begin{tikzpicture}[scale = 1.5, vertex/.style={circle, 
  draw, 
  fill=black,
  inner sep=0.08cm}]
  
  \node [vertex] (a1) at (0,0) {};
  \node [vertex] (a2) at (1,0) {};
  \node [vertex] (a3) at (2,0) {};
  \node [vertex] (a4) at (3,0) {};
  \node [vertex] (a5) at (4,0) {};
  \node [vertex] (a6) at (5,0) {};

	\node (aa1) at (-0.6,0.2) {$a_1 = a_{b_1}$};     	
	\node (aa2) at (1,0.2) {$a_2$};
	\node (aa3) at (2,0.2) {$a_3 = a_{b_2}$};
	\node (aa4) at (3,0.2) {$a_4 = a_{b_3}$};
	\node (aa5) at (4,0.2) {$a_5$};
	\node (aa6) at (5.4,0.2) {$a_6$};
	
	\node [vertex] (c1) at (0,-1) {};
	\node [vertex] (c2) at (2,-1) {};
	\node [vertex] (c3) at (3,-1) {};

	\node (cc1) at (0,-1.2) {$c_1$};
	\node (cc2) at (2,-1.2) {$c_2$};
	\node (cc3) at (3,-1.2) {$c_3$};
	
	\edge{a1}{c1}
	\edge{c1}{a2}
	\edge{a2}{a3}
	\edge{a3}{c2}
	\edge{c2}{a4}
	\edge{a4}{c3}
	\edge{c3}{a5}
	\edge{a5}{a6}
\end{tikzpicture}
\caption{The simple graph $D^*$ corresponding to $D$ in Figure \ref{fig:out-cycle}} \label{fig:D*}
\end{figure}

The analogue of Lemma \ref{lem:A} is proved in the exact same fashion. Lemmas \ref{lem:B} and \ref{lem:C} only involve graph theory and as such we can prove the theorem for co-functional graphs. We now give the analogue of Lemma \ref{lem:D}.
 
\begin{claim} \label{claim:D}
If $H \le D$, then $\stability_+[\loops{H}, 2] \le \stability_+[\loops{D}, 2]$.
\end{claim}

\begin{proof}
Let $f \in \functions_+[\loops{H}, 2]$. Firstly, if $v$ is a source (of $H$ or of $D$, which is equivalent), then $f_v(x) = x_v + \epsilon_v$ for some $\epsilon_v \in [2]$. Secondly, if $u$ is not a source, then $f_u \equiv \phi$ or $f_u \equiv \psi$ for some function $g_u$. We then introduce the function $f' \in \functions_+[\loops{D}, 2]$, defined as
\begin{align*}
	f'_v(x) &= x_v 						&& \text{if } v \text{ is a source},\\
	f'_u &\equiv \begin{cases} 
	\phi & \text{if } f_u \equiv \phi\\
	\psi & \text{if } f_u \equiv \psi
	\end{cases}							&& \text{if } u \text{ is not a source}. 
\end{align*}
(We use the in-neighbourhoods of $\loops{D}$ for $f'$.) Therefore, $\Delta(x, f'(x)) \subseteq \Delta(x, f(x))$ for any $x \in [2]^n$ and $\stability(f') \ge \stability(f)$.
\end{proof}

Finally, combining with Lemma \ref{lem:E}, we prove the theorem for all graphs.
\end{proof}

We now construct an infinite family of loop-full graphs where the strict monotone stability is below $n/2$. Let $m$ be a positive integer, and let $B_m$ be the graph on $n = 2m+1$ vertices with arc set
$$
	E(B_m) = \{(i,i+1) : 1 \le i \le 2m-1 \} \cup \{(2m, 1)\} \cup \{(1,n)\}.
$$
The graph $B_m$ is represented in Figure \ref{fig:Bm}.

\begin{figure}[ht]
\centering
\begin{tikzpicture}[scale = 1.5, vertex/.style={circle, 
  draw, 
  fill=black,
  inner sep=0.08cm}]
  
  \node [vertex] (a1) at (0,0) {};
  \node [vertex] (a2) at (1,0) {};
  \node (a3) at (2,0) {$\cdots$};
  \node [vertex] (a4) at (3,0) {};
  \node [vertex] (a5) at (4,0) {};

	\node (aa1) at (-0.2,0.2) {$1$};     	
	\node (aa2) at (1,0.2) {$2$};
	\node (aa4) at (3,0.2) {$2m-1$};
	\node (aa5) at (4.4,0.2) {$2m$};
	
	\node [vertex] (n) at (0,-1) {};

	\node (nn) at (0,-1.2) {$n$};
	
	\arc{a1}{a2}
	\arc{a2}{a3}
	\arc{a3}{a4}
	\arc{a4}{a5}
	\draw[-latex] (a5) to [bend right] (a1);
	\arc{a1}{c1}
\end{tikzpicture}
\caption{The graph $B_m$.} \label{fig:Bm}
\end{figure}
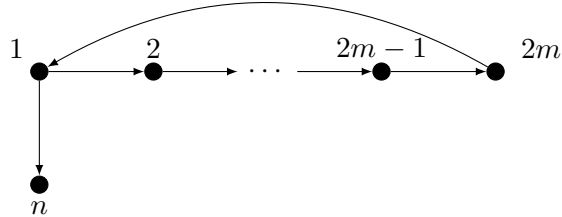

\begin{proposition}
For any $m \ge 1$, $\loops{B_m}$ has strict monotone stability equal to $m$.
\end{proposition}

\begin{proof}
The case $m=1$ is easily verified, thus we assume $m \ge 2$ henceforth. Let $f \in \functions_+[\loops{B_m},2]$ have maximum stability, then every local function $f_v$ is of the form 
$$
	f_v(x) = x_v + \one{ (x_v,x_u) = (\epsilon_v, \epsilon_v + 1) }.
$$
Therefore, $f$ is entirely described by $\epsilon = (\epsilon_1, \dots, \epsilon_n)$. Our aim is to exhibit a state $y \in [2]^n$ such that $\dH(y, f(y)) \ge m+1$. Without loss of generality, let $f_n = \phi_n$, i.e. $\epsilon_n = 0$. We now proceed to a case analysis.

\begin{description}
	\item[Case 1: $\epsilon_2 = 0$.] ~\\
	Let $y$ such that
	$(y_{2k}, y_{2k-1}) = (\epsilon_{2k}, \epsilon_{2k} + 1)$ for all $1 \le k \le m$ and $y_n = 0$.
	Then $\Delta(y, f(y)) = \{2, 4, \dots, 2m, n\}$.

	\item[Case 2: $\epsilon_1 = 1$.] ~\\
	Let $y$ such that
	$(y_{2k+1}, y_{2k}) = (\epsilon_{2k+1}, \epsilon_{2k+1} + 1)$ for all $1 \le k \le m-1$ and $(y_n, y_1, y_{2m}) = (0,1,0)$. Then $\Delta(y, f(y)) = \{1, 3, 5, \dots, 2m-1, n\}$.

\item[Case 3: $\epsilon_1 = 0$ and $\epsilon_2 = 1$.] ~
	\begin{description}
		\item[Case 3.1: For all $i$ from $3$ to $2m$, $\epsilon_i = 1$.] ~\\
		Let $y = (0,1,0,1,\dots,0,1,0)$. Then $\Delta(y, f(y)) = \{1, 2, 4, 6, \dots, 2m\}$.
	
		\item[Case 3.2.: There exists $3 \le i \le 2m$ with $\epsilon_i = 0$.]~\\
		Let  $j := \min\{ 3 \le i \le 2m : \epsilon_j = 0\}$.
		\begin{description}
			\item[Case 3.2.1: $j$ is even.] ~\\
			Let $y$ such that
			$(y_{2k+1}, y_{2k}) = (\epsilon_{2k+1}, \epsilon_{2k+1} + 1)$ for all $1 \le k \le j/2 - 1$, $(y_{2l}, y_{2l-1}) = (\epsilon_{2l}, \epsilon_{2l} + 1)$ for all $j/2 \le l \le m$, and $(y_n, y_1) = (0,1)$. Then $\Delta(y, f(y)) = \{ 3, 5, \dots, j-1, j, j+2, \dots, 2m, n\}$.
			
			\item[Case 3.2.2: $j$ is odd.] ~\\
			Let $y$ such that
			$(y_{2k}, y_{2k-1}) = (1,0)$ for all $1 \le k \le (j-1)/2$, $(y_{2l+1}, y_{2l}) = (\epsilon_{2l+1}, \epsilon_{2l+1} + 1)$ for all $(j-1)/2 \le l \le m-1$, and $(y_{2m}, y_n) = (1,0)$. Then $\Delta(y, f(y)) = \{ 1, 2, 4, \dots, j-1, j, j+2, \dots, 2m-1\}$.
		\end{description}
	\end{description}
\end{description}
\end{proof}

We are now interested in how the monotone stability compares with the stability. Clearly, the following are equivalent: $D$ is the empty graph, $\stability[\loops{D}, 2] = n$ and $\stability_+[\loops{D}, 2] = n$. Once we exclude the empty graph, the two types of stabilities can be as far apart as Theorem \ref{th:strict_monotone_stability_2} allows. The balanced complete bipartite graph $K_{m,m}$ has vertex set $V = L \cup R$, where $L = \{l_1, \dots, l_m\}$ and $R = \{r_1, \dots, r_m\}$, and all possible edges between $L$ and $R$: $E = \{l_i r_j : l_i \in L, r_j \in R\}$.

\begin{proposition}
We have $\stability[\loops{K_{m,m}}, 2] = 2m-1$ while $\stability_+[\loops{K_{m,m}}, 2] = m$.
\end{proposition}

\begin{proof}
Firstly, let $f \in \functions[\loops{K_{n,n}}, 2]$ be defined as
\begin{align*}
	f_{l_i}(x) &= x_{l_i} + \one{ \left( x_{l_i} = 1 \right) \land \left( x_{r_i} = 1 \right) \land \left( x_{R \setminus r_i} = (0,\dots,0) \right) }\\
	f_{r_j}(x) &= x_{r_j} + \one{ \left( x_{r_j} = 1 \right) \land \left( x_{l_j} = 0 \right) \land \left( x_{L \setminus l_j} = (1, \dots, 1) \right) }.
\end{align*}
We now prove that its stability is $2m-1$. If $f_{l_i}(x) \ne x_{l_i}$, then
\begin{itemize}
	\item $x_{l_i} = 1$, hence $f_{r_i}(x) = x_{r_i}$,
	\item $x_{r_i} = 1$ hence $f_{l_k}(x) = x_{l_k}$ for all $k \ne i$,
	\item $x_{R \setminus r_i} = (0, \dots, 0)$ hence $f_{r_k}(x) = x_{r_k}$ for all $k \ne i$.
\end{itemize}
Similarly, if $f_{r_j}(x) \ne x_{r_j}$, then
\begin{itemize}
	\item $x_{r_j} = 1$ hence $f_{l_k}(x) = x_{l_k}$ for all $k \ne j$,
	\item $x_{l_j} = 0$ hence $f_{l_j}(x) = x_{l_j}$ and $f_{r_k}(x) = x_{r_k}$ for all $k \ne j$.
\end{itemize}

Now let $f \in \functions_+[\loops{K_{m,m}}, 2]$ with maximum stability and
\begin{align*}
	A &:= |\{ i : f_{l_i} \equiv \phi\}| + |\{ j : f_{r_j} \equiv \psi\}|\\
	B &:= |\{ i : f_{l_i} \equiv \psi\}| + |\{ j : f_{r_j} \equiv \phi\}|.
\end{align*}
We then have $A + B = 2m$, hence one of the two, say $A$, satisfies $A \ge m$ (the proof is similar if $B \ge m$ instead). Then let $y \in [2]^{2m}$ such that $y_L = (0, \dots, 0)$ and $y_R = (1, \dots, 1)$, then it is clear that $\dH(y, f(y)) = A \ge m$ and hence $\stability(f) \le m$.
\end{proof}

On the other hand, the monotone strict stability can also reach $n-1$, for instance for $D = K_n$; we characterise such graphs below. A biclique is the complement of a bipartite graph. Let us call a loopless graph $D$ a \textbf{near-biclique} if its vertex set can be partitioned into $V = A \cup B \cup S$, where $S$ is the set of sources of $D$, and
\begin{enumerate}
	\item \label{it:1} for all $a_1, a_2 \in A$, either $(a_1, a_2) \in E$ or $(a_2, a_1) \in E$,
	
	\item \label{it:2} similarly for all $b_1, b_2 \in B$, either $(b_1, b_2) \in E$ or $(b_2, b_1) \in E$,
	
	\item \label{it:3} for any $u, v \in V \setminus S$, there exists $w \in (\neighbourhood(u) \cup u) \cap (\neighbourhood(v) \cup v)$,
	
	\item \label{it:4} for any $a \in A$ and $b \in B$, there exists $w \in \neighbourhood(a) \cap \neighbourhood(b)$.
\end{enumerate}

\begin{theorem}
We have $\stability_+[\loops{D}, 2] = n-1$ if and only if $D$ is a non-empty near-biclique.
\end{theorem}

\begin{proof}
All in-neighbourhoods are with respect to $D$. If $D$ is a near-biclique with $A$ and $B$ as above, let $f \in \functions[\loops{D}, 2]$ such that $f_s(x) = x_s$ for all $s \in S$, $f_a \equiv \phi$ for all $a \in A$ and $f_b \equiv \psi$ for all $b \in B$. We now prove that $\stability(f) = n-1$. Suppose that $f_a(x) \ne x_a$ for some $a \in A$ (the proof is similar if $f_b(x) \ne x_b$ for some $b \in B$), then $x_a = 0$ and $x_{\neighbourhood(a)} = (1, \dots, 1)$. For any other $a' \in A$, we have either $(a, a') \in E$ in which case $x_{\neighbourhood(a')} \ne (1, \dots, 1)$, or $(a', a) \in E$ in which case $x_{a'} = 1 \ne 0$; in any case, we have $f_{a'}(x) = x_{a'}$. Now let $b \in B$ and $w \in \neighbourhood(a) \cap \neighbourhood(b)$, then $x_w = 1$, hence $x_{\neighbourhood(b)} \ne (0,\dots,0)$ and $f_b(x) = x_b$.

Conversely, let $f \in \functions_+[\loops{D}, 2]$ with stability $n-1$; we will prove that $D$ must be a near-biclique. Firstly, suppose that $u$ and $v$ are two non-source vertices which form an independent set of size two, then it is easy to check that $f_u \equiv \phi$ and $f_v \equiv \psi$ (or vice versa). Therefore, if $G$ is the simple graph on $V \setminus S$ where $u$ and $v$ are adjacent if and only if they form an independent set in $D$, we see that $G$ must be bipartite. This is equivalent to Properties \ref{it:1} and \ref{it:2}, with $A$ and $B$ forming the bipartition of $G$. Say $f_a \equiv \phi$ for all $a \in A$ and $f_b \equiv \psi$ for all $b \in B$. Secondly, we must have $\sigma(D) = 1$, which ensures that $D$ satisfies Property \ref{it:3}. Finally, suppose there exist $a \in A$ and $b \in B$ with $\neighbourhood(a) \cap \neighbourhood(b) = \emptyset$, then let $x \in [2]^n$ such that $x_a = 0$, $x_{\neighbourhood(a)} = (1, \dots, 1)$, $x_b = 1$, $x_{\neighbourhood(b)} = (0, \dots, 0)$; we see that $f_a(x) \ne x_a$ and $f_b(x) \ne x_b$. Therefore, $D$ must satisfy Property \ref{it:4}.
\end{proof}

\section{Average strict stability and instability} \label{sec:average_i}

\subsection{Average strict stability}

We are first interested in the average value of the stability. The proportion of FDS in $\functions(n,q)$ with positive stability decreases rapidly with $q$. For any set $S \subseteq [q]^n$, the number of functions $f$ in $\functions(n,q)$ such that $x$ and $f(x)$ are at Hamming distance $n$ for all $x \in S$ is $(q-1)^{n|S|} \cdot q^{n(q^n - |S|)}$. We obtain
$$
	|\{ f \in \functions(n,q) : \stability(f) > 0 \}| = q^{nq^n} \left[ 1 - \left( 1 - \frac{1}{q} \right)^n \right]^{q^n}.
$$

Unsurprisingly, the average stability for a given interaction graph tends to zero when $q$ tends to infinity.

\begin{proposition} \label{prop:avg_stability}
For any $D$ and any $q$,
$$
	\frac{1}{|\functions[D,q]|} \sum_{f \in \functions[D,q]} \stability(f) \le \frac{n}{q}.
$$
\end{proposition}

\begin{proof}
Let $y \in [q]^n$. Since $\stability(f) \le n - \dH(y, f(y))$, we have
\begin{align*}
	\sum_{f \in \functions[D,q]} \stability(f) &\le \sum_{d=0}^n (n-d) \left| \left\{ f \in \functions[D,q] : \dH(y, f(y)) = d \right\} \right|\\
	&= \sum_{d=0}^n (n-d) \binom{n}{d} (q-1)^d q^{-n} |\functions[D,q]|,\\
	&= n q^{n-1} q^{-n} |\functions[D,q]|.
\end{align*}
\end{proof}

\subsection{Average strict instability}

The average instability in $\functions[D,q]$ seems to behave in a more complicated fashion, but at least we have a result for the average instability in $\functions(n,q)$.

\begin{proposition} \label{prop:avg_instability}
For any $n \ge 1$,
$$
	\lim_{q \to \infty} \frac{1}{|\functions(n,q)|} \sum_{f \in \functions(n,q)} \instability(f) = e^{-1}.
$$
\end{proposition}

\begin{proof}
Let $q$ be large enough, and suppose $f \in \functions(n,q)$ is chosen uniformly at random. Then
$$
	\Proba\{ \instability(f) = 0 \} = \Proba\{ |\Fix(f)| \ge 1 \} = 1 - e^{-1} + o(1).
$$
Now for any $x \in [q]^n$,
$$
	\Proba\{ \dH(x, f(x)) \ge 2 \} = 1 - \Proba\{ \dH(x, f(x)) \le 1 \} \le 1 - n q^{-n} (q-1) \le 1 - q^{1 - n}.
$$
Thus,
\begin{align*}
	\Proba\{ \instability(f) \ge 2 \} &= \prod_{x \in [q]^n} \Proba\{ \dH(x, f(x)) \ge 2 \}\\
	&\le (1 - q^{1 - n})^{q^n}\\
	&\le e^{-q}.
\end{align*}
Therefore, $P \{ \instability(f) = 1 \} = e^{-1} + o(1)$, and the expected instability tends to $e^{-1}$.
\end{proof}

In order to evaluate the average instability, we study the number of functions with positive instability, or in other words fixed point free functions. For any graph $D$, we partition $\functions[D,q]$ into
\begin{align*}
	\functions_0[D,q] &= \{ f \in \functions[D,q] : |\Fix(f)| = 0 \},\\
	\functions_1[D,q] &= \{ f \in \functions[D,q] : |\Fix(f)| = 1 \},\\
	\functions_2[D,q] &= \{ f \in \functions[D,q] : |\Fix(f)| \ge 2 \}.
\end{align*}
For the sake of conciseness, we use the notation
$$
	p_0[D,q] = \frac{|\functions_0[D,q]|}{|\functions[D,q]|}, \quad p_1[D,q] = \frac{|\functions_1[D,q]|}{|\functions[D,q]|}, \quad p_2[D,q] = \frac{|\functions_2[D,q]|}{|\functions[D,q]|}.
$$
We similarly define $\functions_0(D,q)$, $p_0(D,q)$, etc.

If $D$ is not acyclic, then \cite{AS} gives a construction of a fixed point free element of $\functions[D,q]$ for any $q$. We now give an upper bound on the number of fixed point free functions in $\functions[D,q]$.

\begin{proposition}
For any $D$ and $q$,
$$
	p_0[D,q] \le 1 - q^{- \feedback(D)}.
$$
\end{proposition}

\begin{proof}
We can uniquely express $f \in \functions[D,q]$ as $f(x) = \phi(x) + f(0, \dots, 0)$, where $\phi \in \functions[D,q]$ satisfies $\phi(0, \dots, 0) = (0, \dots, 0)$. The guessing code of $\phi$ is 
$C_\phi := \{ \phi(x) - x : x \in [q]^n\}$ \cite{Gad}. We then have
\begin{align*}
	|C_\phi| &= |\{ y \in [q]^n : \exists x \text{ s.t. } \phi(x) - x = y \}\\
	&= |\{ y \in [q]^n : \exists x \in \Fix(\phi - y) \}\\
	&= |\{f : f(x) - f(0, \dots, 0) = \phi(x), |\Fix(f)| \ge 1\}|.
\end{align*}
Therefore, using the fact that $|C_\phi| \ge q^{n - \feedback(D)}$ \cite{Gad} we obtain
$$
	|\functions_0[D,q]| \ge \left( q^n - q^{n - \feedback(D)} \right) q^{-n} |\functions[D,q]|.
$$
\end{proof}

We remark that the bound above is tight for some cases. Firstly, it is clearly tight if $D$ is acyclic. Secondly, let $q=2$ and $D$ be the disjoint union of $\nu = \feedback(D)$ cycles (with no other arcs). Then $|\functions[D,2]| = 2^n$ for each local function is of the form $f_v(x_u) = x_u + \epsilon_v$, where $\epsilon_v \in \{0,1\}$ and $(u,v) \in D$. A cycle is positive if there is an even number of vertices $v$ such that $\epsilon_v = 1$, and negative otherwise. It is easy to check that either all the cycles of $f$ are positive, in which case $f$ has $2^\nu$ fixed points, or that one of its cycles is negative and hence $f$ has no fixed points. Thus
$$
	|\functions_0[D,2]|  = \left( 1 - 2^{- \feedback(D)} \right) |\functions[D,2]|.
$$

In view of \cite{AS} and \cite[Lemma 1]{GRF15}, the following properties are equivalent.
\begin{enumerate}
	\item $D$ is acyclic.
	
	\item For any $q \ge 2$, $|\functions_0[D,q]| = 0$; in other words, every function with interaction graph $D$ has at least one fixed point.
	
	\item For any $q \ge 2$, $|\functions_2[D,q]| = 0$; in other words, every function with interaction graph $D$ has at most one fixed point.
\end{enumerate}
We are interested in the opposite property, where $|\functions_0[D,q]|$ or $|\functions_2[D,q]|$ grow as a positive proportion of $|\functions[D,q]|$. We remark that since $|\functions[D,q]| \sim |\functions(D,q)|$ for large $q$, the properties below for $\functions[D,q]$ are equivalent to their counterparts for $\functions(D,q)$.

\begin{proposition} \label{prop:properties}
For any graph $D$, consider the following properties.
\begin{enumerate}[label=(\alph*)]
	\item \label{it:d2} There exists a constant $0 < a < 1$ such that for any $q$ large enough,
	$p_2[D,q]| \ge a.$

	\item \label{it:d0} There exists a constant $0 < b < 1$ such that for any $q$ large enough,
	$p_0[D,q] \ge b.$
		
	\item \label{it:d1} There exists a constant $0 < c < 1$ such that for any $q$ large enough,
	$p_1[D,q] \le c.$	

	\item \label{it:i} There exists a constant $d > 0$ such that for any $q$ large enough,
	$$
		\frac{1}{|\functions[D,q]|} \sum_{f \in \functions[D,q]} \instability(f) \ge d.
	$$
\end{enumerate}
Then we have the following implications: \ref{it:d2} $\Rightarrow$ \ref{it:d0} $\Leftrightarrow$ \ref{it:d1} $\Leftrightarrow$ \ref{it:i}.
\end{proposition}

The proof is based on the following lemma: the average number of fixed points in $\functions[D,q]$ is equal to one.

\begin{lemma} \label{lem:average_fixed_points}
For any $D$ and $q$,
$$
	\sum_{f \in \functions[D,q]} |\Fix(f)| = |\functions[D,q]|. 
$$
\end{lemma}

\begin{proof}
Let $\Phi = \{ \phi \in \functions[D, q] : \phi(0, \dots, 0) = (0, \dots, 0) \}$. We can uniquely express any function $f \in \functions[D,q]$ as $f(x) = \phi(x) + f(0, \dots, 0)$, where $\phi \in \Phi$. Then
\begin{align*}
	\sum_{f \in \functions[D,q]} |\Fix(f)| &= \sum_{f \in \functions[D,q]} \sum_{x \in [q]^n} \one{ f(x) = x }\\
	&= \sum_{\phi \in \Phi} \sum_{y \in [q]^n} \sum_{x \in [q]^n} \one{ \phi(x) + y = x }\\
	&= \sum_{\phi \in \Phi} \sum_{x \in [q]^n} \sum_{y \in [q]^n} \one{ y = x - \phi(x) }\\
	&= \sum_{\phi \in \Phi} \sum_{x \in [q]^n} 1\\
	&= |\functions[D,q]|.
\end{align*}
\end{proof}

\begin{proof}[Proof of Proposition \ref{prop:properties}]
\ref{it:d2} $\Rightarrow$ \ref{it:d0}. Suppose that $D$ satisfies Property \ref{it:d2} but not \ref{it:d0}, then for any $a > \epsilon > 0$ there exists $q$ large enough such that $d_2 \ge a$ and $d_0 < \epsilon$, and hence $d_1 + d_2 > 1 - \epsilon$. The average number of fixed points in $\functions[D,q]$ is at least $2 d_2 + d_1 \ge a + 1 - \epsilon > 1$, which contradicts Lemma \ref{lem:average_fixed_points}.

The proof of \ref{it:d1} $\Rightarrow$ \ref{it:d0} is analogous, and hence omitted. Moreover, Property \ref{it:d0} clearly implies \ref{it:d1} for $c = 1 - a$. Finally, the instability of a fixed-point free FDS is between $1$ and $n$, which immediately implies that Properties \ref{it:d0} and \ref{it:i} are equivalent.
\end{proof}

Since the probability that a random mapping in $\functions(n,q)$ has $k$ fixed points tends to $e^{-1}/k!$ for fixed $k$, we see that the graph $\loops{K_n}$ satisfies Property \ref{it:d2}. Moreover, the graph with only $n$ loops satisfies Property \ref{it:d2}, and more precisely
\begin{align*}
	\lim_{q \to \infty} p_0[D,q] &= 1 - (1 - e^{-1})^n,\\
	\lim_{q \to \infty} p_1[D,q] &= e^{-n},\\
	\lim_{q \to \infty} p_2[D,q] &= (1 - e^{-1})^n - e^{-n},\\
	\lim_{q \to \infty} \frac{1}{|\functions[D,q]|} \sum_{f \in \functions[D,q]} \instability(f) &= n e^{-1}.
\end{align*}

We therefore conjecture that if $D$ is non-acyclic, then the average instability is bounded away from zero for large alphabets (Property \ref{it:i}). We give two versions of this conjecture, the first one being stronger than the second according to Proposition \ref{prop:properties}.

\begin{conjecture} \label{conj:d2}
\begin{enumerate}[label=(\alph*)]
	\item Any non-acyclic graph $D$ satisfies Property \ref{it:d2}.
	
	\item Any non-acyclic graph $D$ satisfies Property \ref{it:d0}.
\end{enumerate}
\end{conjecture}

We now prove the conjecture for cycles.

\begin{theorem} \label{th:Cn_property}
For any $n \ge 2$, $\vec{C}_n$ satisfies Property (a). More precisely,
\begin{align*}
	\lim_{q \to \infty} p_0[\vec{C}_n, q] &= e^{-1},\\
	\lim_{q \to \infty} p_1[\vec{C}_n, q] &= e^{-1},\\
	\lim_{q \to \infty} p_2[\vec{C}_n, q] &= 1 - 2e^{-1},\\
	\lim_{q \to \infty} \frac{1}{|\functions[\vec{C}_n, q]|} \sum_{f \in \functions[\vec{C}_n, q]} \instability(f) &= e^{-1}.
\end{align*}
\end{theorem}

\begin{proof}
We have that $S \subseteq [q]^n$ is a set of fixed points of some function $f \in \functions(\vec{C}_n,q)$ if and only if $S$ is a code of minimum distance $n$ \cite{GR11}; in other words, for any distinct $x, y \in S$ and any $v \in V$, $x_v \ne y_v$. The number of codes of minimum distance $n$ and cardinality $t \ge 1$ in $[q]^n$ is then
$$
	c_n := \frac{1}{t!} \left[ q(q-1) \dots (q-t+1) \right]^n.
$$
For any such $S$, let $A_S$ denote the number of functions $f \in \functions(\vec{C}_n,q)$ such that $S \subseteq \Fix(f)$. Then $f \in A_S$ if and only if for any $v \in V$ and any $s \in S$, $f_v(s_{v-1}) = s_v$. Since the values outside of $S$ are arbitrary, there are $q^{q - t}$ choices for $f_v$, and hence
$$
	|A_S| = \prod_{v \in V} q^{q - t} = q^{-nt} |\functions(\vec{C}_n,q)|.
$$

Thus, by the inclusion-exclusion principle, the number of fixed point free functions in $\functions(\vec{C}_n,q)$ is given by
\begin{align*}
	|\functions_0(\vec{C}_n,q)| &= |\functions(\vec{C}_n,q)| \left\{ 1 - \sum_{t=1}^q (-1)^{t-1} c_n q^{-nt} \right\},\\
	p_0(\vec{C}_n,q) &= \sum_{t=0}^q \frac{(-1)^t}{t!} \left[ \left( 1 - \frac{1}{q} \right) \dots \left( 1 - \frac{t-1}{q} \right) \right]^n.
\end{align*}

We denote the right hand side of the last equation as $V_q$.

\begin{claim}
$\lim_{q \to \infty} V_q = e^{-1}$.
\end{claim}

\begin{proof}
Let $\epsilon > 0$ and $q$ large enough so that
$$
	\left| \sum_{t= 0}^L \frac{(-1)^t}{t!} - e^{-1} \right| < \alpha, \quad
	n \frac{L}{q} < \beta, \quad
	\frac{q}{L!} < \gamma,
$$
where $L = \lceil \log q \rceil$ and $\alpha + e \beta + \gamma < \epsilon$. We will prove that $|V_q - e^{-1}| < \epsilon$. Firstly, for all $t \le L$,
$$
	\left[ \left( 1 - \frac{1}{q} \right) \dots \left( 1 - \frac{t-1}{q} \right) \right]^n > \left( 1 - \frac{L}{q} \right)^n > 1 - n \frac{L}{q} > 1 - \beta.
$$
Secondly, 
$$
	\sum_{t=L+1}^q \frac{(-1)^t}{t!} \left[ \left( 1 - \frac{1}{q} \right) \dots \left( 1 - \frac{t-1}{q} \right) \right]^n < \frac{q}{L!} < \gamma.
$$
Combining, we obtain
\begin{align*}
	V_q &< \sum_{\substack{t \text{ even} \\ t \le L}} \frac{1}{t!} - (1 - \beta) \sum_{\substack{t  \text{ odd} \\ t \le L}} \frac{1}{t!} + \gamma\\
	&= \sum_{t=0}^L \frac{(-1)^t}{t!} + \beta \sum_{\substack{t  \text{ odd} \\ t \le L}} \frac{1}{t!} + \gamma\\
	&< e^{-1} + \alpha + e \beta + \gamma \\
	&< e^{-1} + \epsilon.
\end{align*}
On the other hand,
\begin{align*}
	V_q &> (1 - \beta) \sum_{\substack{t \text{ even} \\ t \le L}} \frac{1}{t!} -  \sum_{\substack{t  \text{ odd} \\ t \le L}} \frac{1}{t!}\\
	&= \sum_{t=0}^L \frac{(-1)^t}{t!} - \beta \sum_{\substack{t \text{ even} \\ t \le L}} \frac{1}{t!} \\
	&> e^{-1} - \alpha - e \beta \\
	&> e^{-1} - \epsilon.
\end{align*}
\end{proof}

This proves the limit of $p_0[\vec{C}_n, q]$. Moreover, since any function in $\functions[\vec{C}_n,q]$ has instability at most one, we get the limit of the average instability.

In a similar way, for any $x \in [q]^n$ the number of functions in $\functions_1(\vec{C}_n, q)$ fixing only $x$ is given by
$$
	|\functions(\vec{C}_n, q)| q^{-n} \left\{ 1 - \sum_{t=2}^q \frac{(-1)^t}{(t-1)!} \left[ \left( 1 - \frac{1}{q} \right) \dots \left( 1 - \frac{t-1}{q} \right) \right]^n \right\},
$$
and hence
$$
	p_1(\vec{C}_n, q) = \sum_{s=0}^q \frac{(-1)^s}{s!} \left[ \left( 1 - \frac{1}{q} \right) \dots \left( 1 - \frac{s}{q} \right) \right]^n.
$$
As $q$ tends to infinity, that quantity tends to $V_q$, and hence $p_1[\vec{C}_n,q]$ tends to $e^{-1}$. Finally, we easily obtain the limit of $p_2[\vec{C}_n, q]$. 
\end{proof}

\providecommand{\bysame}{\leavevmode\hbox to3em{\hrulefill}\thinspace}
\providecommand{\MR}{\relax\ifhmode\unskip\space\fi MR }
\providecommand{\MRhref}[2]{%
  \href{http://www.ams.org/mathscinet-getitem?mr=#1}{#2}
}
\providecommand{\href}[2]{#2}

\end{document}